\documentclass[a4paper,12pt]{amsart}
\usepackage{amssymb}
\usepackage{ifthen}
 \usepackage[dvips]{graphicx}
\nonstopmode \numberwithin{equation}{section}
\usepackage{amssymb}
\usepackage{ifthen}
\usepackage{graphicx}
\usepackage{amsmath}
\usepackage[T1]{fontenc} 
\usepackage[utf8]{inputenc}
\usepackage[usenames,dvipsnames]{color}
\usepackage{color}
\usepackage[english]{babel}
\usepackage{fancyhdr}
\usepackage{fancybox}
\usepackage{tikz}

\nonstopmode \numberwithin{equation}{section}
\theoremstyle{plain}

\newtheorem{conj}{Conjecture}

\theoremstyle{definition}
\newtheorem{defn}{Definition}[section]

\newtheorem{thm}{Theorem}[section]
\newtheorem{prob}{Problem}[section]
\newtheorem{cor}{Corollary}[section]

\newtheorem{prop}{Proposition}[section]
\newtheorem{rem}{Remark}[section]
\newtheorem{lem}{Lemma}[section]

\newcounter{minutes}
\divide\time by 60
\newcounter{hours}
\multiply\time by 60
\addtocounter{minutes}{-\time}

\newcounter {own}
\def\theown {\thesection       .\arabic{own}}

\newenvironment{pf}[1][]{%
 \vskip 3mm
 \noindent
 \ifthenelse{\equal{#1}{}}%
  {{\slshape Proof. }}%
  {{\slshape #1.} }%
 }%
{\qed\bigskip}

\newcounter{alphabet}

\def\be{\begin{equation}}
\def\ee{\end{equation}}

\newcommand{\bee}{\begin{enumerate}}
\newcommand{\eee}{\end{enumerate}}

\newcommand{\blem}{\begin{lem}}
\newcommand{\elem}{\end{lem}}
\newcommand{\bthm}{\begin{thm}}
\newcommand{\ethm}{\end{thm}}
\newcommand{\bcor}{\begin{cor}}
\newcommand{\ecor}{\end{cor}}
\newcommand{\beg}{\begin{examp}}
\newcommand{\eeg}{\end{examp}}
\newcommand{\begs}{\begin{examples}}
\newcommand{\eegs}{\end{examples}}

\newcommand{\bdefn}{\begin{defn}}
\newcommand{\edefn}{\end{defn}}

\newcommand{\bprob}{\begin{prob}}
\newcommand{\eprob}{\end{prob}}
\newcommand{\bei}{\begin{itemize}}
\newcommand{\eei}{\end{itemize}}

\newcommand{\bcon}{\begin{conj}}
\newcommand{\econ}{\end{conj}}
\newcommand{\bcons}{\begin{conjs}}
\newcommand{\econs}{\end{conjs}}
\newcommand{\bprop}{\begin{prop}}
\newcommand{\eprop}{\end{prop}}
\newcommand{\br}{\begin{rem}}
\newcommand{\er}{\end{rem}}
\newcommand{\brs}{\begin{rems}}
\newcommand{\ers}{\end{rems}}
\newcommand{\bo}{\begin{obser}}
\newcommand{\eo}{\end{obser}}
\newcommand{\bos}{\begin{obsers}}
\newcommand{\eos}{\end{obsers}}
\newcommand{\bpf}{\begin{pf}}
\newcommand{\epf}{\end{pf}}
\newcommand{\ba}{\begin{array}}
\newcommand{\ea}{\end{array}}
\newcommand{\beq}{\begin{eqnarray}}
\newcommand{\beqq}{\begin{eqnarray*}}
\newcommand{\eeq}{\end{eqnarray}}
\newcommand{\eeqq}{\end{eqnarray*}}

\begin{document}

\title{Generalized weighted composition-differentiation operators on weighted Bergman spaces}

\author{Molla Basir Ahamed}
\address{Molla Basir Ahamed, Department of Mathematics, Jadavpur University, Kolkata-700032, West Bengal, India.}
\email{mbahamed.math@jadavpuruniversity.in}

\author{Taimur Rahman}
\address{Taimur Rahman, Department of Mathematics, Jadavpur University, Kolkata-700032, West Bengal, India.}
\email{taimurr.math.rs@jadavpuruniversity.in}

\subjclass[{AMS} Subject Classification:]{Primary 47B33, 48B38, 30H20, Secondary 47A05, 47B15}

\def\thefootnote{}
\footnotetext{ {\tiny File:~\jobname.tex,
printed: \number\year-\number\month-\number\day,
          \thehours.\ifnum\theminutes<10{0}\fi\theminutes }
} \makeatletter\def\thefootnote{\@arabic\c@footnote}\makeatother

\begin{abstract} 
 Let $ \mathcal{H}(\mathbb{D}) $ be the class of all holomorphic functions in the unit disk $ \mathbb{D} $. We aim to explore the complex symmetry exhibited by generalized weighted composition-differentiation operators, denoted as $L_{n, \psi, \phi}$ and is defined by
\begin{align*}
	L_{n, \psi, \phi}:=\sum_{k=1}^{n}c_kD_{k, \psi_k, \phi},\; \mbox{where }\;  c_k\in\mathbb{C}\; \mbox{for}\;  k=1, 2, \ldots, n,
\end{align*} where $ D_{k, \psi, \phi}f(z):=\psi(z)f^{(k)}(\phi(z)),\; f\in \mathcal{A}^2_{\alpha}(\mathbb{D}),
$ in the reproducing kernel Hilbert space, labeled as $\mathcal{A}^2_{\alpha}(\mathbb{D})$, which encompasses analytic functions defined on the unit disk $\mathbb{D}$. By deriving a condition that is both necessary and sufficient, we provide insights into the $ C_{\mu, \eta} $-symmetry exhibited by $L_{n, \psi, \phi}$. The explicit conditions for which the operator T is Hermitian and normal are obtained through our investigation. Additionally, we conduct an in-depth analysis of the spectral properties of $ L_{n, \psi, \phi} $ under the assumption of $ C_{\mu, \eta} $-symmetry and thoroughly examine the kernel of the adjoint operator of $L_{n, \psi, \phi}$.

\end{abstract}

\maketitle
\pagestyle{myheadings}
\markboth{M. B. Ahamed and T. Rahman}{Generalized weighted composition-differentiation operators on weighted Bergman spaces}
\tableofcontents
\section{Introduction}
Let $ \mathcal{S}=\mathcal{S}(\mathbb{D}) $ be the class of all holomorphic self-maps of the unit disk $ \mathbb{D} $ of the complex plane $ \mathbb{C} $. Each $ \phi\in\mathcal{S} $ induces a composition operator $ \mathcal{C}_{\phi}: \mathcal{H}(\mathbb{D}) \to \mathcal{H}(\mathbb{D}) $ defined by $ \mathcal{C}_{\phi}f:=f \circ \phi $, where $ \mathcal{H}(\mathbb{D}) $ is the class of all holomorphic functions in $ \mathbb{D} $. The theory of composition operators has received significant attention over the past four decades in different contexts. For a detailed understanding of composition operators acting on holomorphic function spaces, we recommend consulting well-known references \cite{Cowen-MacCluer-CRCP-1995} and \cite{Shapiro-SVNY-1993}. The Hardy space denoted by $ \mathcal{H}^2(\mathbb{D}) $, consisting of functions $ f\in\mathcal{H}(\mathbb{D}) $ such that 
\begin{align*}
	||f||^2_{\mathcal{H}^2(\mathbb{D})}=\sup_{0\leq r<1}\frac{1}{2\pi}\int_{0}^{2\pi}|f(re^{i\theta})|^2 d\theta<\infty.
\end{align*}
For $ \alpha>-1 $, the weighted Bergman space $ \mathcal{A}^2_{\alpha}(\mathbb{D}) $ is the set of all functions $ f\in\mathcal{H}(\mathbb{D}) $ for which 
\begin{align*}
	||f||^2_{\mathcal{A}^2_{\alpha}(\mathbb{D})}=\frac{1}{2\pi}\int_{0}^{2\pi}|f(z)|^2 d{A}_{\alpha}(z)<\infty,
\end{align*}
where $ d{A}_{\alpha}(z)=(\alpha+1)(1-|z|^2)^{\alpha}dA(z) $ and $ dA(z)=\frac{r}{\pi}drd\theta$ for $ z=re^{i\theta} $. It is commonly recognized that the weighted Bergman space $ \mathcal{A}^2_{\alpha}(\mathbb{D}) $ constitutes a Hilbert space, wherein the inner product is determined by 
\begin{align*}
	\langle f, g \rangle=\int_{\mathbb{D}}f(w)\overline{g(w)}dA_{\alpha}(w)\; \mbox{for all}\; f, g\in \mathcal{A}^2_{\alpha}(\mathbb{D}).
\end{align*}
The reproducing kernel $ K_w(z) $ of $ \mathcal{A}^2_{\alpha}(\mathbb{D}) $ for the point $ w\in\mathbb{D} $ is 
\begin{align*}
	K_w(z)=\frac{1}{(1-\overline{w}z)^{2+\alpha}}
\end{align*} and a simple computation shows that $ \langle f(z), K_w(z) \rangle=f(w) $ for all $ f\in \mathcal{A}^2_{\alpha}(\mathbb{D}) $ and $ z\in\mathbb{D} $. \vspace{1.2mm}
It is a well-established fact, as shown in \cite[Corollary 3.7]{Cowen-MacCluer-CRCP-1995}, that the composition operator remains bounded when applied to $\mathcal{H}^2(\mathbb{D})$. Specifically, the operator norm of $C_{\varphi}$ satisfies the inequality:
\begin{align*}
	\left(\frac{1}{1-|\varphi(0)|^2}\right)^{1/2} \leq ||C_{\varphi}|| \leq \left(\frac{1+|\varphi(0)|^2}{1-|\varphi(0)|^2}\right)^{1/2}.
\end{align*}
For any analytic function $\psi$, the weighted composition operator $C_{\psi, \phi}: \mathcal{H(\mathbb{D})} \to \mathcal{H(\mathbb{D})}$ is defined as $C_{\psi, \phi}(f) = \psi \cdot (f \circ \phi)$. In the realm of analytic functions, it can be readily verified that the differentiation operator $D(f) = f'$ is not bounded when applied to the Hardy space $\mathcal{H}^2(\mathbb{D})$. This is evident from the fact that the sequence $\{z^n\}_{n=1}^{\infty}$, consisting of unit vectors in $\mathcal{H}^2(\mathbb{D})$, satisfies $||D(f)|| = n$. Nevertheless, for analytic self-mappings $\phi$ defined on the unit disk $\mathbb{D}$, the operator $D_{\phi}: \mathcal{H}^2(\mathbb{D}) \to \mathcal{H}^2(\mathbb{D})$ given by $D_{\phi} = f^{\prime} \circ \phi$ remains bounded. For recent advances, in composition-differentiation operators, we refer to the articles \cite{Han-Wang-BJMA-2021,Han-Wang-CAOT-2021,Han-Wang-BJMA-2022} and references therein.\vspace{1.2mm} 

The study of composition operators over Bergman spaces is an important area of research in complex analysis and operator theory. Bergman spaces are a class of function spaces defined on the unit disk in the complex plane, and composition operators are linear operators that map functions from one space to another by composing them with a fixed function. Bergman spaces are fundamental in the study of function theory, particularly in the context of bounded analytic functions. The composition operators over Bergman spaces provide insights into the behavior and properties of these functions when they are composed with specific functions. Understanding the composition operators helps in investigating the structure and properties of the functions in the Bergman space.\vspace{1.2mm}

The study of composition operators over Bergman spaces is also significant in the realm of operator theory. Composition operators are bounded linear operators that act on function spaces. Analyzing their properties and behavior over Bergman spaces provides insights into their spectral properties, compactness, and other important operator-theoretic properties. The investigation of these properties is crucial for understanding the general theory of operators on function spaces and their applications in various areas of mathematics. The study of composition operators over Bergman spaces has practical applications in engineering and physics. These spaces find applications in signal processing, image analysis, control theory, and quantum mechanics. Understanding the behavior and properties of composition operators over Bergman spaces enables the development of efficient algorithms, signal processing techniques, and mathematical models in various engineering and physical systems.\vspace{1.2mm}

Let $ \phi $ be an analytic self-map of $ \mathbb{D} $ and $ n\in\mathbb{N} $. The composition-differentiation operator $ D_{n, \phi} $ is defined as follows: for any analytic function $ f $ of $ \mathbb{D} $, $ D_{n, \phi}f(z):=f^{(n)}(\phi(z)) $. Moreover, let $ \psi $ be an analytic function defined on $ \mathbb{D} $, the weighted composition-differentiation operator, denoted by $ D_{n, \psi, \phi} $, is defined as follows:
\begin{align*}
	 D_{n, \psi, \phi}f(z)=\psi(z)f^{(n)}(\phi(z)),\; f\in \mathcal{H}(\mathbb{D}).
\end{align*}
\begin{rem}
	The following observations are clear.
	\begin{enumerate}
		\item[(i)] If $ n=0 $, then the operator becomes the weighted composition operator $ D_{\psi, \phi} $. 
		\item[(ii)] In particular, if $ \psi(z)\equiv 1 $ and $ n=1 $, then the operator $ D_{n, \psi, \phi} $ becomes the composition operator $ C_{\phi}D $.
		\item[(iii)] For $ \psi(z)=\phi^{\prime}(z) $ and $ n=1 $, the operator in fact reduces to the product $ DC_{\phi} $, where $ C_{\phi}f(z)=D_{0, \phi}f(z)=f(\phi(z)) $.
	\end{enumerate}
\end{rem}
For more details about the operators $ C_{\phi}D $ and $ DC_{\phi} $ and their different properties, we refer to the articles \cite{Ohno-BAMS-2006,Ohno-BKMS-2009,Stevic-Sharma-Bhat-AMC-2011} and references therein. In a recent publication \cite{Fatehi-Hammond-CAOT-2021}, Fatehi and Hammond undertook an analysis of the properties of weighted composition-differentiation operators applied to the Hardy space . This particular space, denoted as $ \mathcal{H}^2(\mathbb{D}) $, comprises all analytic functions on the unit disk $ \mathbb{D} $ that possess a power series representation with complex coefficients whose squares have a finite sum.\vspace{1.2mm}

In this paper, our objective is to investigate various properties of weighted composition differentiation operators for weighted Bergman spaces, expanding the scope of study beyond existing results. Henceforward, we denote $ L_{n, \psi, \phi} $ by composition linear differentiation operator and is defined by
\begin{align*}
	L_{n, \psi, \phi}:=\sum_{k=1}^{n}c_kD_{k, \psi_k, \phi},\; \mbox{where }\;  c_k\in\mathbb{C}\; \mbox{for}\;  k=1, 2, \ldots, n.
\end{align*}

The outline of this paper is organized as follows. In Section 2, we explicitly determine the forms of $ \psi_k $ (where $ k=1, 2, \ldots, n $) and $ \phi $ such that the generalized weighted composition-differentiation operator $ L_{n, \psi, \phi} $ exhibits complex symmetry on $ \mathcal{A}^2_{\alpha}(\mathbb{D}) $ with conjugation $ C_{\mu, \eta} $. Furthermore, we establish the necessary and sufficient conditions for $ L_{n, \psi, \phi} $ to be normal and Hermitian. In Section 3, we investigate the kernel of the adjoint operator of $ L_{n, \psi, \phi} $ and offer the corresponding kernel under the assumption of complex symmetry with conjugation $ C_{\mu, \eta} $. In Section 4, we analyze the spectral characteristics of $ L_{n, \psi, \phi} $ and determine its spectrum under the conditions of complex symmetry with conjugation $ C_{\mu, \eta} $ and $ \phi(0)=0 $.
\section{Complex symmetric generalized weighted composition-differentiation operator on the Bergman space $ \mathcal{A}^2_{\alpha}(\mathbb{D}) $}

In this section, we mainly discuss complex symmetry, normality and self-adjointness of the generalized weighted composition-differentiation operator $ L_{n, \psi, \phi} $ on $ \mathcal{A}^2_{\alpha}(\mathbb{D}) $.\vspace{1.2mm}

Suppose $ n $ is a non-negative integer and $ w $ belongs to the set $ \mathbb{D} $. We define $ K^{[n]}_w(z) $ as the reproducing kernel for point-evaluation of the $ n $-th derivative, with
\begin{align*}
	  K^{[n]}_w(z)=\frac{p_nz^n }{(1-\bar{w}z)^{n+\alpha+2}}, z\in\mathbb{D},
\end{align*}
where $p_n=(\alpha+2)(\alpha+3) \ldots(\alpha+n+1)$. It can be proven that $ \langle f(z),  K^{[n]}_w(z)\rangle=f^{(n)}(w) $, for all $ f\in\mathcal{A}^2_{\alpha}(\mathbb{D}) $. For any non-negative integer $ n $, we define 
\begin{align*}
	\gamma_n=\sqrt{\frac{\Gamma(n+\alpha+2)}{n!\Gamma(2+\alpha)}}z^n, z\in\mathbb{D},
\end{align*}
where $ \Gamma $ is the standard Gamma function . Then the set $ \{\gamma_n:n\ge0\} $ forms an orthonormal basis for $ \mathcal{A}^2_{\alpha}(\mathbb{D}) $ (see\cite{Hedenmalm-Korenblum-Zhu-GTMS-2000}).
\vspace{1.2mm} 

An antilinear map $ \mathcal{C}:\mathcal{H}\to\mathcal{H} $ on a separable Hilbert space $ \mathcal{H} $ is called a conjugation if it satisfies $ \langle \mathcal{C}x,\mathcal{C}y\rangle=\langle y,x\rangle $ for all $ x,y\in\mathcal{H} $, and $ \mathcal{C}^2=I $. A bounded linear operator $T$ on $\mathcal{H}$ is said to be complex symmetric if there exists a conjugation $ \mathcal{C} $ on $ \mathcal{H} $ such that $ T=\mathcal{C}T^*\mathcal{C} $.\vspace{1.2mm} 

For $\mu, \eta $ on the unit circle ${z\in\mathbb{C}:|z|=1}$, the conjugation $ \mathcal{C}_{\mu, \eta} $ is defined as $ \mathcal{C}_{\mu, \eta}f=\mu\overline{f(\overline{\eta z})}$, where $f$ belongs to the class of analytic functions. In this paper, we explore the complex symmetric weighted composition-differentiation operator on the weighted Bergman space $ \mathcal{A}^2_{\alpha}(\mathbb{D}) $ with respect to the conjugation $ \mathcal{C}_{\mu, \eta} $. \vspace{1.2mm}

The recent work by Garcia and Hammond \cite{Garcia-Hammond-OPAP-2013} focused on the characteristics of a complex symmetric weighted composition operator on weighted Hardy spaces, where they provided explicit formulas for these operators using a particular conjugation technique. From \cite{Allu-Hal-Pal-2023}, we have a lemma that directly addresses the behavior of the adjoint of the weighted composition-differentiation operator $ D_{n, \psi, \phi} $ with respect to the reproducing kernel of $ \mathcal{A}^2_{\alpha}(\mathbb{D}) $.
\begin{lem}\cite{Allu-Hal-Pal-2023}\label{lem-2.1}
Let $ \phi $ be an analytic self-map of $ \mathbb{D} $ and $ \psi\in\mathcal{H}(\mathbb{D})$  such that $ D_{n, \psi, \phi} $ is bounded on $ \mathcal{A}^2_{\alpha}(\mathbb{D}) $. Then for any $ w\in\mathbb{D} $, 
\begin{align*}
	D^*_{n, \psi, \phi}K_{w}(z)=\overline{\psi(w)}K^{[n]}_{\phi(w)}.
\end{align*}	
\end{lem}
In order to make progress with our research, it is imperative to formulate a lemma that broadens the applicability of Lemma \ref{lem-2.1} in a more generalized way. As a result, we present the following lemma, which specifically addresses the action of the adjoint of the generalized composition-differentiation operator $ L_{n, \psi, \phi} $ on the reproducing kernel of $ \mathcal{A}^2_{\alpha}(\mathbb{D}) $.

\begin{lem}\label{lem-2.2}
	Let $ \phi $ be an analytic self-map of $ \mathbb{D} $ and $ \psi_k\in\mathcal{H}(\mathbb{D}) $ ($ k=1, 2, \ldots, k $) such that $ L_{n, \psi, \phi}:=\sum_{k=1}^{n}c_kD_{k, \psi, \phi} $ (where $ c_k\in\mathbb{C} $ for $ k=1, 2, \ldots, n $) is bounded on $ \mathcal{A}^2_{\alpha}(\mathbb{D}) $. Then for any $ w\in\mathbb{D} $, 
	\begin{align*}
		L^*_{n, \psi, \phi}K_{w}(z)=\sum_{k=1}^{n}\overline{c_k\psi_k(w)}K^{[k]}_{\phi(w)}.
	\end{align*}
\end{lem}
\begin{proof}[\bf Proof of Lemma \ref{lem-2.2}]
	Let $ f\in\mathcal{A}^2_{\alpha}(\mathbb{D}) $. Then a simple computation shows that 
	\begin{align*}
		\langle f,  L^*_{n, \psi, \phi}K_{w}\rangle_{\mathcal{A}^2_{\alpha}(\mathbb{D})}&=\langle L_{n, \psi, \phi},  K_{w}\rangle_{\mathcal{A}^2_{\alpha}(\mathbb{D})}\\ &=\bigg\langle \sum_{k=1}^{n}c_kD_{k, \psi_k, \phi}, K_{w} \bigg\rangle_{\mathcal{A}^2_{\alpha}(\mathbb{D})}\\&=\sum_{k=1}^{n}\bigg\langle c_k\psi_k\left(f^{(k)}\circ\phi\right), K_w \bigg\rangle_{\mathcal{A}^2_{\alpha}(\mathbb{D})}\\&=\sum_{k=1}^{n} c_k\psi_k(w)f^{(k)}(\phi(w))\\&=\sum_{k=1}^{n}\bigg\langle f,\;  \overline{c_k\psi_k(w)}K^{[k]}_{\phi(w)} \bigg\rangle_{\mathcal{A}^2_{\alpha}(\mathbb{D})}\\&=\bigg\langle f,\;  \sum_{k=1}^{n}\overline{c_k\psi_k(w)}K^{[k]}_{\phi(w)} \bigg\rangle_{\mathcal{A}^2_{\alpha}(\mathbb{D})}.
	\end{align*} 
	Hence, we see that 
	\begin{align*}
		\langle f,  L^*_{n, \psi, \phi}K_{w}\rangle_{\mathcal{A}^2_{\alpha}(\mathbb{D})}=\bigg\langle f,\;  \sum_{k=1}^{n}\overline{c_k\psi_k(w)}K^{[k]}_{\phi(w)} \bigg\rangle_{\mathcal{A}^2_{\alpha}(\mathbb{D})}\;\mbox{for all}\; f\in\mathcal{A}^2_{\alpha}(\mathbb{D})
	\end{align*}
	which leads the following conclusion
	\begin{align*}
		L^*_{n, \psi, \phi}K_{w}(z)=\sum_{k=1}^{n}\overline{c_k\psi_k(w)}K^{[k]}_{\phi(w)}.
	\end{align*}
	This completes the proof.
\end{proof}
In this section, we aim to explore the conditions under which certain combinations of $\psi_k$ ($k=1, 2, \ldots,n) $ and $\phi$ produce $C_{\mu,\eta}$-symmetric generalized weighted composition-differentiation operators $ L_{n, \psi, \phi} $. In the following result, we obtain a necessary and sufficient condition for which the generalized weighted composition-differentiation operator $ L_{n, \psi, \phi} $ is complex symmetric on  $ \mathcal{A}^2_{\alpha}(\mathbb{D}) $ with conjugation  $ C_{\mu, \eta} $.  
\begin{thm}\label{thm-2.1}
	Let $ \phi : \mathbb{D}\to \mathbb{D} $ be an analytic self-map of $ \mathbb{D} $ and $ \psi_k\in\mathcal{H} $ with $ \psi_k\neq 0 $ ($ k=1, 2, \ldots, n$) such that $ L_{n, \psi, \phi} $ is bounded on $ \mathcal{A}^2_{\alpha}(\mathbb{D}) $. Then $ L_{n, \psi, \phi} $ is complex symmetric on $ \mathcal{A}^2_{\alpha}(\mathbb{D}) $ with conjugation $ C_{\mu, \eta} $ if, and only if, 
	\begin{align*}
		\psi_k(z)=\frac{a_kz^k}{(1-\eta bz)^{k+\alpha+2}}\; ({k=1, 2, \ldots, n})\; \mbox{and}\; \phi(z)=b+\frac{cz}{1-\eta bz}\; \mbox{for all}\;  z\in\mathbb{D},
	\end{align*}
	where $ c, a_1, \ldots,a_k\in\mathbb{C} $ and $ b\in\mathbb{D} $. 
\end{thm}
In view of Theorem \ref{thm-2.1}, it is easy to see that, in particular, if $ c_1=c_2=\cdots=c_{n-1}=0 $ and $ c_n=1 $, then $ L_{n, \psi, \phi} $ becomes $ D_{n, \psi, \phi} $. Consequently,  the following result can be easily obtained.

\begin{cor}\cite[Therem 2.1]{Allu-Hal-Pal-2023}
Let $ \phi : \mathbb{D}\to \mathbb{D} $ be an analytic self-map of $ \mathbb{D} $ and $ \psi\in\mathcal{H} $ with $ \psi\neq 0 $ such that $ D_{n, \psi, \phi} $ is bounded on $ \mathcal{A}^2_{\alpha}(\mathbb{D}) $. Then $ D_{n, \psi, \phi} $ is complex symmetric on $ \mathcal{A}^2_{\alpha}(\mathbb{D}) $ with conjugation $ C_{\mu, \eta} $ if, and only if, 
\begin{align*}
	\psi(z)=\frac{az^n}{(1-\eta bz)^{n+\alpha+2}}\; \mbox{and}\; \phi(z)=b+\frac{cz}{1-\eta bz}\; \mbox{for all}\;  z\in\mathbb{D},
\end{align*}
where $ a, c\in\mathbb{C} $ and $ b\in\mathbb{D} $. 	
\end{cor}
The following result can be readily derived based on the implications of Theorem \ref{thm-2.1}.
\begin{cor}\label{cor-2.2}
	Let $ \phi : \mathbb{D}\to \mathbb{D} $ be an analytic self-map of $ \mathbb{D} $ and $ \psi\in\mathcal{H} $ with $ \psi\neq 0 $ such that $ D_{\psi, \phi} $ is bounded on $ \mathcal{A}^2_{\alpha}(\mathbb{D}) $. Then $ D_{\psi, \phi} $ is complex symmetric on $ \mathcal{A}^2_{\alpha}(\mathbb{D}) $ with conjugation $ C_{\mu, \eta} $ if, and only if, 
	\begin{align*}
		\psi(z)=\frac{az}{(1-\eta bz)^{\alpha+3}}\; \mbox{and}\; \phi(z)=b+\frac{cz}{1-\eta bz}\; \mbox{for all}\;  z\in\mathbb{D},
	\end{align*}
	where $ a, c\in\mathbb{C} $ and $ b\in\mathbb{D} $.
\end{cor}
 
We see that Corollary \ref{cor-2.2} can be considered an analogue of \cite[Theorem 2.1]{Liu-Ponnusamy-Xie-LMA-2023}, as both offer similar characterizations.

\begin{proof}[\bf Proof of Theorem \ref{thm-2.1}]
Suppose that $ L_{n, \psi, \phi} $ be complex symmetric with conjugation $ C_{\mu, \eta} $. Then, we have
	\begin{align}\label{Eq-2.1}
		L_{n, \psi, \phi}C_{\mu, \eta}K_w(z)=C_{\mu, \eta}L^*_{n, \psi, \phi}K_w(z)\;\mbox{for all} z,w\in\mathbb{D}.
	\end{align}
	Moreover, an easy computation shows that
	\begin{align}\label{Eq-2.2}
		L_{n, \psi, \phi}C_{\mu, \eta}K_w(z)&=L_{n, \psi, \phi}C_{\mu, \eta}\bigg(\frac{1}{(1-\overline{w}z)^{\alpha+2}}\bigg)\\&=L_{n, \psi, \phi}\bigg(\frac{\mu}{(1-\eta wz)^{\alpha+2}}\bigg)\nonumber\\&=\sum_{k=1}^{n}c_kD_{k, \psi_k, \phi}\bigg(\frac{\mu}{(1-\eta wz)^{\alpha+2}}\bigg)\nonumber\\&=\sum_{k=1}^{n}\frac{c_k p_k\mu\psi_k(z) (\eta w)^{k}}{(1-\eta w\phi(z))^{k+\alpha+2}}\nonumber.
	\end{align}
	and 
	\begin{align}\label{Eq-2.3}
		C_{\mu, \eta}L^*_{n, \psi, \phi}K_w(z)&=C_{\mu, \eta}\sum_{k=1}^{n }\overline{c_k\psi_k(w)}K^{[k]}_{\phi(w)}(z)\\&=C_{\mu, \eta}\sum_{k=1}^{n}\overline{c_k\psi_k(w)}\frac{p_kz^k}{\bigg(1-\overline{\phi(w)}z\bigg)^{k+\alpha+2}}\nonumber\\&=\sum_{k=1}^{n}\frac{c_kp_k\mu\psi_k(w)(\eta z)^k}{(1-\phi(w)\eta z)^{k+\alpha+2}}\nonumber.
	\end{align}
	By using \eqref{Eq-2.2} and \eqref{Eq-2.3} in \eqref{Eq-2.1}, we have
	
	\begin{align*}
		\sum_{k=1}^{n}\frac{c_k p_k\mu\psi_k(z) (\eta w)^{k}}{(1-\eta w\phi(z))^{k+\alpha+2}}&=\sum_{k=1}^{n}\frac{c_kp_k\mu\psi_k(w)(\eta z)^k}{(1-\phi(w)\eta z)^{k+\alpha+2}}\; \mbox{for all}\; z,w\in\mathbb{D}
	\end{align*}
As $ \mu ,\eta\in\{z\in\mathbb{C}:|z|=1\} $, the above equation implies that 
\begin{align}\label{Eq-2.4}
	\frac{\psi_k(z)w^{k}}{(1-\eta w\phi(z))^{k+\alpha+2}}=\frac{\psi_k(w)z^k}{(1-\phi(w)\eta z)^{k+\alpha+2}}
\end{align}
for $ k=1, 2, \ldots,n $ and for all $ z,w\in\mathbb{D} $.\vspace{1.2mm}

	Letting $ z=0 $ in \eqref{Eq-2.4}, we obtain $ \psi_k(0) =0$ for $ k=1, 2, \ldots, n. $ Let $ \psi_k(z)=z^mg_k(z) $, where $ m\in\mathbb{N} $ and $ g_k $ is analytic on $ \mathbb{D} $ with $ g_k(0)\neq0 $. Now, our claim is $ m=k $.\vspace{1.2mm}
	
	\noindent {\bf Case I}: If $ m > k $, from \eqref{Eq-2.4} it follows that 
	\begin{align*}
		\frac{z^{m-k}g_k(z)}{(1-\eta w\phi(z))^{k+\alpha+2}}=\frac{w^{m-k}g_k(w)}{(1-\phi(w)\eta z)^{k+\alpha+2}}	
	\end{align*}
	for any $w,z\in\mathbb{D} $. Putting $ w=0 $ in the above equation, we obtain $ g_k(z)=0 $ on $\mathbb{D}$, which is a contradiction of the fact $ g_k(0)\neq0 $.\vspace{1.2mm}
	
	\noindent {\bf Case II}: If $ m<k $, from \eqref{Eq-2.4} we obtain
	\begin{align*}
		\frac{w^{k-m}g_k(z)}{(1-\eta w\phi(z))^{k+\alpha+2}}=\frac{z^{k-m}g_k(w)}{(1-\phi(w)\eta z)^{k+\alpha+2}}	
	\end{align*}
for any $ w, z\in\mathbb{D} $. Setting $ w=0 $ in preceding equation gives $ g_k(0)=0 $, which contradicts the assumption  that $ g_k(0)\neq0 $.\vspace{1.2mm}

Therefore, we must have $ m=k $ and hence \eqref{Eq-2.4} is equivalent to 
	\begin{align}\label{Eq-2.5}
		\frac{g_k(z)}{(1-\eta w\phi(z))^{k+\alpha+2}}=\frac{g_k(w)}{(1-\phi(w)\eta z)^{k+\alpha+2}}	
	\end{align}
	for all $ z,w \in\mathbb{D} $.\vspace{1.2mm} 
	
	By letting $ w=0 $ in \eqref{Eq-2.5}, we obtain
	\begin{align*}
		g_k(z)=\frac{g_k(0)}{(1-\phi(0)\eta z)^{k+\alpha+2}}\; \mbox{for}\; k=1,2, \ldots, n.	
	\end{align*}
	Thus, we have 
	\begin{align*}
		\psi_k(z)=\frac{a_kz^k}{(1-\eta bz)^{k+\alpha+2}}\; \mbox{for}\; k=1,2 \ldots,n
	\end{align*}
where $ a_k=g_k(0) (k=1, 2, \ldots, n) $ and $ b=\phi(0) $.\vspace{1.2mm} 

Substituting $\psi_k(z) $ into \eqref{Eq-2.4}, we obtain 
	\begin{align}\label{Eq-2.6}
		(1-\eta bz)^{k+\alpha+2}(1-\eta w\phi(z))^{k+\alpha+2}=(1-\eta bw)^{k+\alpha+2}(1-\eta\phi(w)z)^{k+\alpha+2}
	\end{align}
	for all $ z,w\in\mathbb{D} $. Differentiating both sides of \eqref{Eq-2.6} with respect to $ w $, we obtain 
	\begin{align*}
		(k+\alpha+2)(1-&\eta bz)^{k+\alpha+2}(1-\eta w\phi(z))^{k+\alpha+1}(-\eta \phi(z))\\&=(k+\alpha+2)(1-\eta bw)^{k+\alpha+1}(1-\eta\phi(w)z)^{k+\alpha+2}(-\eta b)\\&\quad+(k+\alpha+2)(1-\eta bw)^{k+\alpha+2}(1-\eta\phi(w)z)^{k+\alpha+1}(-\eta z\phi'(w)).	
	\end{align*} 
	Letting $ w=0 $ in the above equation, we obtain 
	\begin{align*}
		\phi(z)=b+\frac{cz}{1-\eta bz}
	\end{align*}
	for all $ z\in\mathbb{D} $, where $ c=\phi'(0) $.\vspace{1.2mm}
	
	Conversely, let $ a_1, a_2, \ldots,a_k,c\in\mathbb{C} $ and $ b\in\mathbb{D} $ be such that 
	\begin{align*}
		\psi_k(z)=\frac{a_kz^k}{(1-\eta bz)^{k+\alpha+2}}\; \mbox{for}\; k=1, 2, \ldots,n\; \mbox{and}\; \phi(z)=b+\frac{cz}{1-\eta bz}\; \mbox{for all}\;  z\in\mathbb{D}.
	\end{align*}
In view of \eqref{Eq-2.2} and \eqref{Eq-2.3}, we have
	\begin{align}\label{Eq-2.7}
		L_{n, \psi, \phi}C_{\mu, \eta}K_w(z)&=\sum_{k=1}^{n}\frac{c_k p_k\mu\psi_k(z) (\eta w)^{k}}{(1-\eta w\phi(z))^{k+\alpha+2}}\\&=\sum_{k=1}^{n}\frac{c_k p_ka_k\mu (\eta wz)^{k}}{(1-\eta bz)^{k+\alpha+2}\bigg(1-\eta w(b+\frac{cz}{1-\eta bz})\bigg)^{k+\alpha+2}}\nonumber\\&=\sum_{k=1}^{n}\frac{c_k p_ka_k\mu (\eta wz)^{k}}{(1-\eta bz-\eta bw+\eta^2 b^2wz-\eta cwz)^{k+\alpha+2}}\nonumber.	
	\end{align}
and 
	\begin{align}\label{Eq-2.8}
		C_{\mu, \eta}L^*_{n, \psi, \phi}K_w(z)&=	\sum_{k=1}^{n}\frac{c_kp_k\mu\psi_k(w)(\eta z)^k}{(1-\phi(w)\eta z)^{k+\alpha+2}}\\&=\sum_{k=1}^{n}\frac{c_kp_ka_k\mu(\eta wz)^k}{(1-\eta bw)^{k+\alpha+2}(1-\eta z(b+\frac{cw}{1-\eta bw}))^{k+\alpha+2}}\nonumber\\&=\sum_{k=1}^{n}\frac{c_kp_ka_k\mu(\eta wz)^k}{(1-\eta bz-\eta bw+\eta^2 b^2wz-\eta cwz)^{k+\alpha+2}}\nonumber.
	\end{align}
	It follows from \eqref{Eq-2.7} and \eqref{Eq-2.8} that 
	\begin{align*}
		L_{n, \psi, \phi}C_{\mu, \eta}K_w(z)=C_{\mu, \eta}L^*_{n, \psi, \phi}K_w(z)\; \mbox{for all}\; z\in\mathbb{D}.
	\end{align*}
	Since the span of the reproducing kernel functions is dense in $ \mathcal{A}^2_{\alpha}(\mathbb{D}) $, the operator $ L_{n, \psi, \phi} $  is complex symmetric with conjugation $ C_{\mu, \eta} $. This completes the proof.
\end{proof}
Next, we proceed to establish a sufficient condition that guarantees the normality of the bounded operator $ L_{n, \psi, \phi} $, given the presence of complex symmetry through conjugation $ C_{\mu, \eta} $. This result expands upon and strengthens the findings of \cite[Theorem 2.3]{Allu-Hal-Pal-2023}.
\begin{thm}\label{thm-2.2}
	Let $ \phi $ be an analytic self-map of $ \mathbb{D} $ with $ \phi(0)=0 $ and $ \psi_k\in\mathcal{H}(\mathbb{D})$, where $ k=1, 2, \ldots, n $,  be not identically zero such that for any $ n\in\mathbb{N} $, the operator $ L_{n, \psi, \phi} $ is bounded and complex symmetric with conjugation $ C_{\mu, \eta} $. Then  $ L_{n, \psi, \phi} $ is normal. 
\end{thm}
The consequences of Theorem \ref{thm-2.2} indicate that if $ c_1=c_2=\ldots=c_{n-1}=0 $ and $ c_n=1 $, then it is evident that $ L_{n, \psi, \phi}=D_{n, \psi, \phi} $. Consequently, the following result can be obtained easily from Theorem \ref{thm-2.2}.
\begin{cor}\cite[Theorem 2.3]{Allu-Hal-Pal-2023}
	Let $ \phi $ be an analytic self-map of $ \mathbb{D} $ with $ \phi(0)=0 $ and $ \psi\in\mathcal{H}(\mathbb{D}) $ be not identically zero such that for any $ n\in\mathbb{N} $, the operator $ D_{n, \psi, \phi} $ is bounded and complex symmetric with conjugation $ C_{\mu, \eta} $. Then  $ D_{n, \psi, \phi} $ is normal.
\end{cor}
It is easy to see that the Theorem \ref{thm-2.2} leads us to a straightforward derivation of the subsequent result.
\begin{cor}\label{cor-2.4}
Let $ \phi $ be an analytic self-map of $ \mathbb{D} $ with $ \phi(0)=0 $ and $ \psi\in\mathcal{H}(\mathbb{D}) $ be not identically zero such that the operator $ D_{\psi, \phi} $ is bounded on $ \mathcal{A}^2_{\alpha}(\mathbb{D}) $ and complex symmetric with conjugation $ C_{\mu, \eta} $. Then  $ D_{\psi, \phi} $ is normal.	
\end{cor}
By the hypothesis of the Corollary \ref{cor-2.4}, we can easily obtain from Corollary \ref{cor-2.2} that $\psi(z)=az$ and $\phi(z)=cz$, where $a, c\in\mathbb{C}$. Thus The Corollary \ref{cor-2.4} can be seen as analogous to the sufficient part of  \cite[Theorem 2.4]{Liu-Ponnusamy-Xie-LMA-2023}.
\begin{proof}[\bf Proof of Theorem \ref{thm-2.2}]
	Since  $ L_{n, \psi, \phi} $ is $ C_{\mu, \eta} $-symmetric on $ \mathcal{A}^2_{\alpha}(\mathbb{D}) $ and  $ \phi(0)=0 $, in view of Theorem $ \ref{thm-2.1} $, we have 
	\begin{align*}
		\psi_k(z)=a_kz^k\; \mbox{for}\; k=1, 2, \ldots, n\; \mbox{and}\; \phi(z)=cz,
	\end{align*} 
where $ a_1, a_2, \ldots, a_n, c\in\mathbb{C} $.\vspace{1.2mm}

	Then for any non-negative integer $ k $, a simple computation shows that
	\begin{align*}
		\lVert L_{n, \psi, \phi}\gamma_m \rVert^2&=\sum_{j=0}^{\infty}\lvert\langle L_{n, \psi, \phi}\gamma_m, \gamma_j\rangle\rvert^2\\&=\sum_{j=0}^{\infty}\bigg\lvert\bigg\langle \sum_{k=1}^{n}c_kD_{k, \psi_k, \phi}\gamma_m, \gamma_j\bigg\rangle\bigg\rvert^2\\&=\sum_{j=0}^{\infty}\bigg\lvert\bigg\langle \sum_{k=1}^{n}c_k\psi_k\gamma_m^{(k)}(\phi), \gamma_j\bigg\rangle\bigg\rvert^2\\&=\sum_{j=0}^{\infty}\bigg\lvert\bigg\langle \sum_{k=1}^{n}\frac{c_ka_k m!c^{m-k}}{(m-k)!}\sqrt{\frac{\Gamma(m+\alpha+2)}{m!\Gamma(\alpha+2)}}z^m, \sqrt{\frac{\Gamma(j+\alpha+2)}{j!\Gamma(\alpha+2)}}z^j\bigg\rangle\bigg\rvert^2.
	\end{align*}
	Moreover, on the other hand, we have
	\begin{align*}
		\lVert L^*_{n, \psi, \phi}\gamma_m \rVert^2&=\sum_{j=0}^{\infty}\lvert\langle L^*_{n, \psi, \phi}\gamma_m, \gamma_j\rangle\rvert^2\\&=\sum_{j=0}^{\infty}\lvert\langle \gamma_m, L_{n, \psi, \phi}\gamma_j\rangle\rvert^2\\&=\sum_{j=0}^{\infty}\bigg\lvert\bigg\langle \gamma_m, \sum_{k=1}^{n}c_kD_{k, \psi_k, \phi}\gamma_j\bigg\rangle\bigg\rvert^2\\&=\sum_{j=0}^{\infty}\bigg\lvert\bigg\langle \gamma_m, \sum_{k=1}^{n}c_k\psi_k\gamma_j^{(k)}(\phi)\bigg\rangle\bigg\rvert^2\\&=\sum_{j=0}^{\infty}\bigg\lvert\bigg\langle \sqrt{\frac{\Gamma(m+\alpha+2)}{m!\Gamma(\alpha+2)}}z^m, \sum_{k=1}^{n}\frac{c_ka_k j!c^{j-k}}{(j-k)!}\sqrt{\frac{\Gamma(j+\alpha+2)}{j!\Gamma(\alpha+2)}}z^j\bigg\rangle\bigg\rvert^2	
	\end{align*}
	Since the set $ \{\gamma_n:n\ge 0\} $ forms an orthonormal basis for $ \mathcal{A}^2_\alpha(\mathbb{D}) $, thus it follows that for any $ k\in\mathbb{N}\cup\{0\} $
	\begin{align*}
		\lVert L_{n, \psi, \phi}\gamma_m \rVert^2=\bigg\lvert\sum_{k=1}^{n}\frac{c_ka_km!c^{m-k}}{(m-k)!}\bigg\rvert^2=\lVert L^*_{n, \psi, \phi}\gamma_m \rVert^2.
	\end{align*}
	Consequently, in view of the above estimates, it follows that 
	\begin{align*}
		\lVert L_{n, \psi, \phi}(f) \rVert=\lVert L^*_{n, \psi, \phi}(f) \rVert\; \mbox{for all}\; f\in\mathcal{A}^2_{\alpha}(\mathbb{D}), 
	\end{align*}
	and $  L_{n, \psi, \phi} $ is a normal operator. This completes the proof.
\end{proof}
It should be emphasized that Theorem \ref{thm-2.2} does not encompass all normal weighted composition-differentiation operators on $ \mathcal{A}^2_{\alpha}(\mathbb{D}) $. It is of significance to highlight that we can precisely determine the normality of an operator $ L_{n, \psi, \phi} $ when $ \psi_k $ ($ k=1, 2, \ldots, n $) and $ \phi $ adopt a specific form that generalizes the symbols presented in the proof of Theorem \ref{thm-2.2}. Despite our efforts, we have not yet been able to establish whether an operator $ L_{n, \psi, \phi} $, where $ \phi(0)\neq 0 $, can exhibit normal behavior for different choices of $ \psi_k $ ($ k=1, 2, \ldots, n $) and $ \phi $.\vspace{1.2mm}

We know that a bounded linear operator $ T $ is said to be Hermitian( self-adjoint) if $ T^*=T $. In recent years, significant research has been focused on investigating the self-adjointness of weighted composition-differentiation operators acting on the reproducing kernel Hilbert space (see \cite{ Fatehi-Moradi-2021, Lim-Khoi-JMAA-2018, Lo-Loh-JMAA-2023} and references therein). In  \cite{Fatehi-Hammond-CAOT-2021}, Fatehi and Hummond provide a complete characterization of self-adjointness of weighted composition-differentiation operator on the Hardy space. In \cite{Liu-Ponnusamy-Xie-LMA-2023}, Liu \emph{et al.} have discussed the self-adjointness of the weighted composition-differentiation operator on the Bergman space $ \mathcal{A}^2_{\alpha}(\mathbb{D}) $. In the following result, we establish a condition that is both necessary and sufficient for the bounded operator $ L_{n, \psi, \phi} $ to satisfy Hermitian properties on $ \mathcal{A}^2_{\alpha}(\mathbb{D}) $. 
\begin{thm}\label{thm-2.3}
	Let $ \phi $ be an analytic self-map of $\mathbb{D} $ and $ \psi_k $ (where $ k=1, 2, \ldots, n $) be nonzero analytic functions on the unit disc such that $ L_{n, \psi, \phi}=\sum_{k=1}^{n}c_kD_{k, 
		\psi_k, \phi} $ is bounded on $ \mathcal{A}^2_{\alpha}(\mathbb{D}) $ and $ c_1, c_2, \ldots, c_n\in\mathbb{R} $. Then $ L_{n, \psi, \phi} $ is Hermitian if, and only if,
	\begin{align*}
		\psi_k(z)=\frac{a_kz^k}{(1-\overline{b}z)^{k+\alpha+2}}\; \mbox{for}\; k=1, 2, \ldots ,n\; \mbox{and}\; \phi(z)=b+\frac{cz}{(1-\overline{b}z)}
	\end{align*}
	for some $ a_1, a_2, \ldots, a_n, c\in\mathbb{R} $ and $ b\in\mathbb{D} $.
\end{thm}
Based on the implications of Theorem \ref{thm-2.3}, it is straightforward to derive the following results.
\begin{cor}\cite[Theorem 2.4]{Allu-Hal-Pal-2023}
		Let $ \phi $ be an analytic self-map of $\mathbb{D} $ and $ \psi $ be a nonzero analytic functions on the unit disk $\mathbb{D} $  such that $ D_{n, \psi, \phi}$ is bounded on $ \mathcal{A}^2_{\alpha}(\mathbb{D}) $. Then $ D_{n, \psi, \phi} $ is Hermitian if, and only if,
	\begin{align*}
		\psi(z)=\frac{az^n}{(1-\overline{b}z)^{n+\alpha+2}}\; \mbox{and}\; \phi(z)=b+\frac{cz}{(1-\overline{b}z)}
	\end{align*}
	for some $ a, c\in\mathbb{R} $ and $ b\in\mathbb{D} $.
\end{cor}
\begin{cor}\label{cor-2.5}
	Let $ \phi $ be an analytic self-map of $\mathbb{D} $ and $ \psi $ be a nonzero analytic functions on the unit disk $\mathbb{D} $  such that $ D_{\psi, \phi}$ is bounded on $ \mathcal{A}^2_{\alpha}(\mathbb{D}) $. Then $ D_{\psi, \phi} $ is Hermitian if, and only if,
	\begin{align*}
		\psi(z)=\frac{az}{(1-\overline{b}z)^{\alpha+3}}\; \mbox{and}\; \phi(z)=b+\frac{cz}{(1-\overline{b}z)}
	\end{align*}
	for some $ a, c\in\mathbb{R} $ and $ b\in\mathbb{D} $.
\end{cor}
It is worth noticing that Corollary \ref{cor-2.5} and \cite[Theorem 2.3]{Liu-Ponnusamy-Xie-LMA-2023} have a similar characterization, making them analogues of each other.
\begin{proof}[\bf Proof of Theorem \ref{thm-2.3}]
	Let $ L_{n, \psi, \phi} $ be Hermitian. Then we have $ L^*_{n, \psi, \phi} $ = $ L_{n, \psi, \phi} $, which is equivalent to
	\begin{align}\label{Eq-2.9}
		L^*_{n, \psi, \phi}K_w(z)=L_{n, \psi, \phi}K_w(z)\; \mbox{for all}\; z\in\mathbb{D}.
	\end{align}
Using Lemma \ref{lem-2.1}, a simple computation shows that
	\begin{align}\label{eq-2.10}
		L^*_{n, \psi, \phi}K_w(z)=\sum_{k=1}^{n}\overline{c_k\psi_k(w)}K^{[k]}_{\phi(w)}=\sum_{k=1}^{n}\overline{c_k\psi_k(w)}\frac{p_kz^k}{(1-\overline{\phi(w)}z)^{k+\alpha+2}}
	\end{align}
	and
	\begin{align}\label{Eq-2.11}
		L_{n, \psi, \phi}K_w(z)=\sum_{k=1}^{n}c_kD_{k, 
			\psi_k, \phi}K_w(z)=\sum_{k=1}^{n}c_k\psi_k(z)\frac{p_k\overline{w}^k}{(1-\overline{w}\phi(z))^{k+\alpha+2}}.
	\end{align} 
	Substituting \eqref{eq-2.10} and \eqref{Eq-2.11} into \eqref{Eq-2.9}, we obtain
	\begin{align*}
		\sum_{k=1}^{n}\overline{c_k\psi_k(w)}\frac{p_kz^k}{(1-\overline{\phi(w)}z)^{k+\alpha+2}}=\sum_{k=1}^{n}c_k\psi_k(z)\frac{p_k\overline{w}^k}{(1-\overline{w}\phi(z))^{k+\alpha+2}}.
	\end{align*} 
	This implies that 
	\begin{align}\label{Eq-2.12}
		\frac{z^k\overline{\psi_k(w)}}{(1-\overline{\phi(w)}z)^{k+\alpha+2}}=\frac{\overline{w}^k\psi_k(z)}{(1-\overline{w}\phi(z))^{k+\alpha+2}}\; \mbox{for all}\; k=1, 2, \ldots,n
	\end{align}
	for all $ z,w\in\mathbb{D} $.\vspace{1.2mm} 
	
	Letting $ w=0 $ in the above equality \eqref{Eq-2.12}, we obtain  $ \psi_k(0)=0 $ for all $ k=1, 2, \ldots,n $. For $ \psi_k(z)\in\mathcal{A}^2_{\alpha}(\mathbb{D}) $, if $ \psi_k(z)=z^mg_k(z) $, where $ m\in\mathbb{N} $ and $ g_k $ is analytic with $ g_k(0)\neq 0 $, then we aim to show that $ m=k $.\\
	
	\noindent {\bf Case I}: If $ m>k $, then it follows from \eqref{Eq-2.12} that
	\begin{align}\label{Eq-2.13}
		\frac{(\overline{w})^{m-k}\overline{g_k(w)}}{(1-\overline{\phi(w)}z)^{k+\alpha+2}}=\frac{z^{m-k}g_k(z)}{(1-\overline{w}\phi(z))^{k+\alpha+2}}
	\end{align}
	for all $ z,w\in\mathbb{D} $. Setting $ w=0 $ in \eqref{Eq-2.13}, we obtain $ g_k(z)=0 $ on $ \mathbb{D} $, which is a contradiction to the assumption that $ g_k(0)=0 $.\\ 
	
	\noindent {\bf Case II}: If $ m<k $, then from \eqref{Eq-2.12}, we obtain 
	\begin{align}\label{eq-2.14}
		\frac{z^{k-m}\overline{g_k(w)}}{(1-\overline{\phi(w)}z)^{k+\alpha+2}}=\frac{(\overline{w})^{k-m}g_k(z)}{(1-\overline{w}\phi(z))^{k+\alpha+2}}	
	\end{align}
	for all $ z,w\in\mathbb{D} $. Putting $ w=0 $ in \eqref{eq-2.14}, we obtain $ g_k(0)=0 $, which is a contradiction. Therefore, we must have $ m=k $. Now, \eqref{Eq-2.12} reduces to
	\begin{align}\label{Eqq-2.14}
		\frac{\overline{g_k(w)}}{(1-\overline{\phi(w)}z)^{k+\alpha+2}}=\frac{g_k(z)}{(1-\overline{w}\phi(z))^{k+\alpha+2}}
	\end{align}
	for all $ z,w\in\mathbb{D} $. Putting $ w=0 $ in \eqref{Eqq-2.14}, we obtain 
	\begin{align*}
		g_k(z)=\frac{\overline{g_k(0)}}{(1-\overline{\phi(0)}z)^{k+\alpha+2}}\; \mbox{for}\; k=1, 2, \ldots,n.
	\end{align*}
	Thus, we have
	\begin{align}\label{2.13}
		\psi_k(z)=\frac{a_kz^k}{(1-\overline{b}z)^{k+\alpha+2}}\; \mbox{where}\; a_k=\overline{g_k(0)}\; \mbox{and}\; b=\phi(0)\; \mbox{for}\; k=1, 2, \ldots,n.
	\end{align}
	Substituting \eqref{2.13} into \eqref{Eq-2.12}, we obtain
	\begin{align}\label{Eq-2.14}
		a_k(1-b\overline{w})^{k+\alpha+2}(1-\overline{\phi(w)}z)^{k+\alpha+2}=\overline{a_k}(1-\overline{b}z)^{k+\alpha+2}(1-\overline{w}\phi(z))^{k+\alpha+2}
	\end{align} 
	for all $ z,w\in\mathbb{D} $. If we set $ w=0 $ in \eqref{Eq-2.14}, then it is easy to see that $ a_k=\overline{a_k} $, and hence $ a_k\in\mathbb{R} $ for all $ k=1, 2, \ldots,n $.\vspace{1.2mm} 
	
	Differentiating both sides of \eqref{Eq-2.14} with respect to $ \overline{w} $, we obtain 
	\begin{align}\label{Eq-2.15}
		(1-&\overline{b}z)^{k+\alpha+2}(1-\overline{w}\phi(z))^{k+\alpha+1}\phi(z)=b(1-b\overline{w})^{k+\alpha+1}(1-\overline{\phi(w)}z)^{k+\alpha+2}\\&\quad+(1-b\overline{w})^{k+\alpha+2}(1-\overline{\phi(w)}z)^{k+\alpha+1}(\overline{z\phi^{\prime}(w)})\; \mbox{for all}\; z\in\mathbb{D}\nonumber.
	\end{align}
	Setting $ w=0 $, we easily obtain from \eqref{Eq-2.15} that
	\begin{align*}
		\left(1-\overline{\phi(0)}z\right)^{k+\alpha+2}\phi(z)=\left(1-\overline{\phi(0)}z\right)^{k+\alpha+2}b+\left(1-\overline{\phi(0)}z\right)^{k+\alpha+1}\overline{\phi^{\prime}(0)}z
	\end{align*}
	which implies that 
	\begin{align}\label{Eq-1.19}
		\phi(z)=b+\frac{\phi^{\prime}(0)z}{1-\overline{\phi(0)}z}=b+\frac{cz}{1-\bar{b}z},
	\end{align}
	where $ \overline{\phi^{\prime}(0)}=c $ and $ \phi(0)=b $.\\
	
	Differentiating \eqref{Eq-1.19}, we obtain $ \phi^{\prime}(0)=c $. Therefore, we see that $ c=\bar{c} $ which implies that $ c $ is a real number. \vspace{1.2mm}
	
	Conversely, we assume that 
	\begin{align*}
		\psi_k(z)=\frac{a_kz^k}{(1-\bar{b}z)^{k+\alpha+2}}\; \mbox{for}\;  k=1, 2, \ldots, n\; \mbox{and}\; \phi(z)=b+\frac{cz}{1-\bar{b}z},
	\end{align*}
	where $ a_1, a_2, \ldots, a_n\in\mathbb{R} $.\vspace{2mm}
	
	Then, in view of \eqref{eq-2.10} and \eqref{Eq-2.11}, a straightforward computation gives that 
	\begin{align}\label{Eq-2.16}
		L_{n, \psi, \phi}K_w(z)&=\sum_{k=1}^{n}c_k\psi_k(z)\frac{p_k\overline{w}^k}{(1-\overline{w}\phi(z))^{k+\alpha+2}}\\&=\sum_{k=1}^{n}a_kc_k\left(\frac{z^k}{(1-\bar{b}z)^{k+\alpha+2}}\right)\frac{p_k\overline{w}^k}{\left(1-\overline{w}\left(b+\frac{cz}{1-\bar{b}z}\right)\right)^{k+\alpha+2}}\nonumber\\&=\sum_{k=1}^{n}\frac{c_ka_kp_k(z\overline{w})^k}{\left(1-\overline{b}z-b\overline{w}+\overline{w}|b|^2z-\overline{w}cz\right)^{k+\alpha+2}}\nonumber
	\end{align}
	and
	\begin{align}\label{eq-2.17}
		L^*_{n, \psi, \phi}K_w(z)&=\sum_{k=1}^{n}\overline{c_k\psi_k(w)}\frac{p_kz^k}{(1-z\overline{\phi(w)})^{k+\alpha+2}}\\&=\sum_{k=1}^{n}\overline{a_kc_k}\left(\frac{\overline{w}^k}{(1-{b}\overline{w})^{k+\alpha+2}}\right)\frac{p_kz^k}{\left(1-z\left(\overline{b}+\dfrac{\overline{cw}}{1-\bar{b}\overline{w}}\right)\right)^{k+\alpha+2}}\nonumber\\&=\sum_{k=1}^{n}\frac{\overline{c_ka_k}p_k(z\overline{w})^k}{\left(1-\overline{b}z-b\overline{w}+\overline{w}|b|^2z-\overline{w}cz\right)^{k+\alpha+2}}\nonumber.
	\end{align}
	Since $ a_k $ and $ c_k $ are reals, hence $ \overline{a_kc_k}=a_kc_k $. Therefore, from \eqref{Eq-2.16} and \eqref{eq-2.17} we see that $ L^*_{n, \psi, \phi}K_w(z)=L_{n, \psi, \phi}K_w(z) $ for  all $ z\in\mathbb{D} $. Since the space of reproducing kernel function is dense in $ \mathcal{A}^2_{\alpha}(\mathbb{D}) $, it follows that $ L^*_{n, \psi, \phi}=L_{n, \psi, \phi} $  which shows that $ L_{n, \psi, \phi} $ is Hermitian. This completes the proof.
\end{proof}
\section{Kernel of generalized composition-differentiation operators on weighted Bergman space $ \mathcal{A}^2_{\alpha}(\mathbb{D}) $}
 The kernel of linear operator T on a Hilbert space $ \mathcal{H} $ is denoted by $ ker(T) $ and is defined as $ ker(T)=\{x\in\mathcal{H}:T(x)=0\} $. The investigation on finding the kernel of composition operators precisely on the domain of analytic functions holds great importance as it introduces a novel aspect to operator theory. For instance, it follows from \cite[Proposition 3.2]{Jun-Kim-Ko-Lee-JFA-2014} that the kernel of the adjoint of a complex
symmetric weighted composition operator on $ {H}^2(\mathbb{D}) $ is $ \{0\} $. Moreover, the result \cite[Proposition 2.9]{Han-Wang-BJMA-2021} shows that the kernel of the adjoint of a nonzero $ J $-symmetric weighted composition–
differentiation operator on Hardy Hilbert space $ {H}^2(\mathbb{D}) $ is the whole complex plane $ \mathbb{C} $. We obtain the following result exploring that the kernel of the adjoint of a nonzero  $ C_{\mu, \eta} $-symmetric generalized composition-differentiation operator $ L_{n,\psi, \phi}=\sum_{k=1}^{n}c_kD_{k, \psi_k, \phi} $ on the Weighted Bergman Space  $ \mathcal{A}^2_{\alpha}(\mathbb{D}) $ is $ \mathbb{P}_{m-1}(\mathbb{C}) $, the set of all polynomials of degree less than or equals to $(m-1)$, where  $ m=\min\{k:c_k\neq0\}	$.
\begin{thm}\label{thm-3.1}
Let $ \phi $ be an analytic self-map of $ \mathbb{D} $ and $ \psi_k\in\mathcal{H}(\mathbb{D}) $ (where $ k=1, 2, \ldots, n $) be not identically zero such that $ L_{n,\psi, \phi}=\sum_{k=1}^{n}c_kD_{k, \psi_k, \phi} $ is bounded on $ \mathcal{A}^2_{\alpha}(\mathbb{D}) $. If  $ L_{n,\psi, \phi} $ is complex symmetric on $ \mathcal{A}^2_{\alpha}(\mathbb{D}) $ with conjugation $ C_{\mu, \eta} $, then $ ker(L_{n,\psi, \phi})=ker(L^*_{n,\psi, \phi})=\mathbb{P}_{m-1}(\mathbb{C}) $, where $ m=min\{k:c_k\neq0\}	$.
\end{thm}
\begin{proof}[\bf Proof of Theorem \ref{thm-3.1}]
	Suppose $ L_{n,\psi, \phi} $ is complex symmetric on $ \mathcal{A}^2_{\alpha}(\mathbb{D}) $ with conjugation $ C_{\mu, \eta} $, it follows from Theorem \ref{thm-3.1} that
	\begin{align*}
		\psi_k(z)=\frac{a_kz^k}{(1-\eta bz)^{k+\alpha+2}}\; \mbox{for}\; k=1, 2, \ldots,n\; \mbox{and}\; \phi(z)=b+\frac{cz}{1-\eta bz}\; \mbox{for all}\;  z\in\mathbb{D}
	\end{align*}
where $ a_k={\psi^{(k)}(0)}/{k!}\neq 0 $ (where $ k=1, 2, \ldots, n $), $ b=\phi(0) $ and $ c=\phi^{\prime}(0)\neq 0 $. Let $ f\in\ ker(L_{n,\psi, \phi}) $. Then it is easy to see that
\begin{align*}
&L_{n,\psi, \phi}f(z)\equiv0\; \mbox{on}\; \mathbb{D}\\\Rightarrow&\sum_{k=1}^{n}c_kD_{k, \psi_k, \phi}f(z)\equiv0\; \mbox{on}\; \mathbb{D}\\\Rightarrow&\sum_{k=m}^{n}c_k\psi_k(z)f^{(k)}(\phi(z))\equiv0\; \mbox{on}\; \mathbb{D}\\\Rightarrow&f^{(k)}(\phi(z))\equiv0\; \mbox{on}\; \mathbb{D}\; \mbox{for all}\; k=m, m+1, \ldots, n.	
\end{align*}
Since $ \phi(\mathbb{D}) $ is open, by Identity Theorem for complex valued functions, we have
\begin{align*}
	&f^{(k)}(z)\equiv0\; \mbox{on}\; \mathbb{D}\; \mbox{for all}\; k=m, m+1, \ldots, n\\\Rightarrow&f(z)=a_0+a_1z+a_2z^2+ \ldots+ a_{k-1}z^{k-1}.
\end{align*} 
for some  $ a_0,a_1,a_2, \ldots, a_{k-1}\in\mathbb{C} $. Thus it is clear that
\begin{align}\label{Eq-3.1}
ker(L_{n,\psi, \phi})\subseteq\mathbb{P}_{k-1}(\mathbb{C}).	
\end{align}
Let $ f\in\mathbb{P}_{m-1}(\mathbb{C}) $. Clearly, $ f $ is a polynomial of degree $ \leq(m-1) $. A simple computation shows that
\begin{align*}
L_{n,\psi, \phi}f(z)&=\sum_{k=1}^{n}c_kD_{k, \psi_k, \phi}f(z)\\&=\sum_{k=m}^{n}c_k\psi_k(z)f^{(k)}(\phi(z))\nonumber\\&=\sum_{k=m}^{n}c_k\psi_k(z)f^{(k)}(z)\; \mbox{on}\; \mathbb{D}\nonumber\\&\equiv 0\; \mbox{on}\; \mathbb{D}\nonumber.	
\end{align*}
This implies that 
\begin{align}\label{Eq-3.2}
	\mathbb{P}_{m-1}(\mathbb{C})\subseteq ker(L_{n,\psi, \phi}).
\end{align}
Consequently, from \eqref{Eq-3.1} and \eqref{Eq-3.2}, we have $ ker(L_{n,\psi, \phi})=\mathbb{P}_{m-1}(\mathbb{C}) $. Let  $ f\in\ ker(L^*_{n,\psi, \phi}) $. This shows that $ L^*_{n,\psi, \phi}f\equiv 0 $. Moreover, $ L_{n,\psi, \phi} $ is complex symmetric on $ \mathcal{A}^2_{\alpha}(\mathbb{D}) $ with conjugation $ C_{\mu, \eta} $ we must have $ L_{n,\psi, \phi}C_{\mu, \eta}f=C_{\mu, \eta}L^*_{n,\psi, \phi}f\equiv 0 $. This implies that 
\begin{align*}
C_{\mu, \eta}f\in\ker( L_{n,\psi, \phi})=	\mathbb{P}_{m-1}(\mathbb{C})	
\end{align*}
Clearly, $ f\in\ker( L_{n,\psi, \phi}) $. Therefore, it is easy to see that 
\begin{align}\label{Eq-3.3}
	ker( L^*_{n,\psi, \phi})\subseteq\mathbb{P}_{m-1}(\mathbb{C}). 
\end{align}
Let $ g\in\mathbb{P}_{m-1}(\mathbb{C}) $. Thus, we have
\begin{align*}
L^*_{n,\psi, \phi}g=C_{\mu, \eta}L_{n,\psi, \phi}C_{\mu, \eta}g\equiv 0	
\end{align*}
which  implies that $ g\in\ker(L^*_{n,\psi, \phi}) $. Clearly, we have
\begin{align}\label{Eq-3.4}
\mathbb{P}_{m-1}(\mathbb{C})\subseteq ker(L^*_{n,\psi, \phi}).	
\end{align} 
Consequently, from \eqref{Eq-3.3} and \eqref{Eq-3.4}, we obtain $ ker(L^*_{n,\psi, \phi})=\mathbb{P}_{m-1}(\mathbb{C}) $. Thus, we have
\begin{align*}
ker(L_{n,\psi, \phi})=ker(L^*_{n,\psi, \phi})=	\mathbb{P}_{m-1}(\mathbb{C}).
\end{align*}
This completes the proof.
\end{proof}
As a consequence of Theorem \ref{thm-2.1}, we obtain the following result for generalized composition-differentiation operators $ D_{n,\psi, \phi} $ for the space $ \mathcal{A}^2_{\alpha}(\mathbb{D}) $. The following result implies that the kernel of the adjoint of a nonzero $ C_{\mu, \eta} $-symmetric generalized weighted composition-differentiation operator $ D_{n,\psi, \phi} $ on $ \mathcal{A}^2_{\alpha}(\mathbb{D}) $ is $ \mathbb{P}_{n-1}(\mathbb{C}) $. 
 \begin{cor}\label{cor-3.1}
Let $ \phi $ be an analytic self-map of $ \mathbb{D} $ and $ \psi\in\mathcal{H}(\mathbb{D}) $ be not identically zero such that $ D_{n,\psi, \phi} $ is bounded on $ \mathcal{A}^2_{\alpha}(\mathbb{D}) $. If  $ D_{n,\psi, \phi} $ is complex symmetric on $ \mathcal{A}^2_{\alpha}(\mathbb{D}) $ with conjugation $ C_{\mu, \eta} $, then $ ker(D_{n,\psi, \phi})=ker(D^*_{n,\psi, \phi})=\mathbb{P}_{n-1}(\mathbb{C}) $.
 \end{cor}
The following result implies that the kernel of the adjoint of a nonzero weighted differentiation-composition operator on $ \mathcal{A}^2_{\alpha}(\mathbb{D}) $ is the whole complex plane $ \mathbb{C} $. 
\begin{cor}\label{cor-3.2}
Let $ \phi $ be an analytic self-map of $ \mathbb{D} $ and $ \psi\in\mathcal{H}(\mathbb{D}) $ be not identically zero such that $ D_{1,\psi, \phi}=D_{\psi, \phi} $ is bounded on $ \mathcal{A}^2_{\alpha}(\mathbb{D}) $. If  $ D_{\psi, \phi} $ is complex symmetric on $ \mathcal{A}^2_{\alpha}(\mathbb{D}) $ with conjugation $ C_{\mu, \eta} $, then $ ker(D_{\psi, \phi})=ker(D^*_{\psi, \phi})=\mathbb{C} $.
\end{cor}
 \begin{proof}[\bf Proof of Corollary \ref{cor-3.1}]
 	Suppose that $ D_{n,\psi, \phi} $ is complex symmetric on $ \mathcal{A}^2_{\alpha}(\mathbb{D}) $ with conjugation $ C_{\mu, \eta} $. It follows from \cite[Theorem-2.1]{Allu-Hal-Pal-2023} that
 	\begin{align*}
 		\psi(z)=\frac{az^n}{(1-\eta bz)^{n+\alpha+2}}\; \mbox{and}\; \phi(z)=b+\frac{cz}{1-\eta bz}	
 	\end{align*}
 	where $ a={\psi^{(n)}(0)}/{n!}\neq 0 $, $ b=\phi(0) $ and $ c=\phi^{\prime}(0)\neq 0 $. Let $ f\in\ ker(D_{n,\psi, \phi}) $. Then 
 	\begin{align*}
 		D_{n,\psi, \phi}(f(z))\equiv 0\; \mbox{on}\; \mathbb{D}.\\
 		\Rightarrow\psi(z)f^{(n)}(\phi(z))\equiv 0\; \mbox{on}\; \mathbb{D}.\\\Rightarrow f^{(n)}(\phi(z))\equiv 0\; \mbox{on}\; \mathbb{D}.
 	\end{align*}
 	Since $ \phi(\mathbb{D}) $ is open, by the Identity Theorem we have 
 	\begin{align*}
 		f^{(n)}(z)&\equiv 0\; \mbox{on}\; \mathbb{D}.\\\Rightarrow f(z)&=a_0+a_1z+a_2z^2+ \ldots+ a_{n-1}z^{n-1}
 	\end{align*}
 	for some $ a_0, a_1, a_2, \ldots, a_{n-1}\in\mathbb{C} $. Clearly, it follows that
 	\begin{align}\label{Eq-3.5}
 		ker(D_{n,\psi, \phi})\subseteq\mathbb{P}_{n-1}(\mathbb{C}) 
 	\end{align}
 	Let $ f\in\mathbb{P}_{n-1}(\mathbb{C}) $. It is easy to see that $ f $ is a polynomial of degree $ \leq(n-1) $. An easy computation yields that
 	\begin{align*}
 		D_{n,\psi, \phi}(f(z))&=\psi(z)f^{(n)}(\phi(z))\\&\equiv\psi(z)f^{(n)}(z)\; \mbox{on}\; \mathbb{D}\\&\equiv 0\; \mbox{on}\; \mathbb{D}
 	\end{align*}
 	which turns out that $ f\in\ker(D_{n,\psi, \phi}) $. Thus, we have
 	\begin{align}\label{Eq-3.6}
 		\mathbb{P}_{n-1}(\mathbb{C})\subseteq\ker(D_{n,\psi, \phi})
 	\end{align}
 	consequently, in view of \eqref{Eq-3.5} and \eqref{Eq-3.6}, we obtain $ ker(D_{n,\psi, \phi})=\mathbb{P}_{n-1}(\mathbb{C}) $. Furthermore, for $ f\in\ ker(D^*_{n,\psi, \phi}) $, we have $	D^*_{n,\psi, \phi}(f)=0 $. Since $ D_{n,\psi, \phi} $ is complex symmetric on $ \mathcal{A}^2_{\alpha}(\mathbb{D}) $ with conjugation $ C_{\mu, \eta} $, a straightforward computation shows that
 	\begin{align*}
 		&D_{n,\psi, \phi}C_{\mu, \eta}f=C_{\mu, \eta}D^*_{n,\psi, \phi}f=0\\\Rightarrow &C_{\mu, \eta}f\in\ker(D_{n,\psi, \phi})=\mathbb{P}_{n-1}(\mathbb{C})\\\Rightarrow&\mu\overline{f(\overline{\eta z})}\in\mathbb{P}_{n-1}(\mathbb{C})\\\Rightarrow&f\in\mathbb{P}_{n-1}(\mathbb{C})	
 	\end{align*}
 Hence,
 	\begin{align}\label{Eq-3.7}
 		ker(D^*_{n,\psi, \phi})\subseteq\mathbb{P}_{n-1}(\mathbb{C}). 	
 	\end{align} 
 	Let $ g\in\mathbb{P}_{n-1}(\mathbb{C}) $. Then, we have $	D^*_{n,\psi, \phi}g=C_{\mu, \eta}D_{n,\psi, \phi}C_{\mu, \eta}g=0 $ which implies that $ g\in\ker(D^*_{n,\psi, \phi}) $. Hence,
 	\begin{align}\label{Eq-3.8}
 		\mathbb{P}_{n-1}(\mathbb{C})\subseteq ker(D^*_{n,\psi, \phi}). 	
 	\end{align}
 	Consequently, from \eqref{Eq-3.7} and \eqref{Eq-3.8}, we obtain $ ker(D^*_{n,\psi, \phi})=\mathbb{P}_{n-1}(\mathbb{C}) $. Therefore, we conclude that
 	\begin{align*}
 		ker(D_{n,\psi, \phi})=ker(D^*_{n,\psi, \phi})=\mathbb{P}_{n-1}(\mathbb{C}).	
 	\end{align*}
 	This completes the proof.
 \end{proof}
\section{Spectrum of generalized weighted composition-differential operators on space $ \mathcal{A}^2_{\alpha}(\mathbb{D}) $ }
In this section, we consider spectral properties of weighted composition-differentiation operators. The study of spectral properties of composition-differentiation operators is focused on understanding the behavior and characteristics of these operators in relation to their eigenvalues and eigenvectors. These operators combine the operations of composition and differentiation on functions, resulting in a new transformed function.\vspace{2mm}

Investigating the spectral properties of composition-differentiation operators is significant for several reasons. Firstly, it provides insights into the structure and behavior of these operators, shedding light on their fundamental properties and how they affect the functions they operate on. This knowledge can be applied in various areas of mathematics and mathematical analysis. Secondly, understanding the spectral properties of composition-differentiation operators enables us to analyze and solve differential equations involving these operators. Differential equations play a crucial role in many scientific and engineering disciplines, and by studying the spectral properties, we can gain a deeper understanding of the solutions to such equations.\vspace{2mm}

Moreover, the study of spectral properties also has connections to harmonic analysis, functional analysis, and operator theory. It allows us to explore the connections between composition-differentiation operators and other important mathematical concepts, leading to further advancements and applications in these fields. Overall, investigating the spectral properties of composition-differentiation operators provides valuable insights into their behavior, facilitates the analysis of differential equations, and establishes connections with other branches of mathematics, contributing to the broader understanding and application of these operators in various mathematical contexts.\vspace{1.2mm}

Let $ \mathcal{B}(\mathcal{H}) $ be the algebra of all bounded linear operators on a separable complex Hilbert space $ \mathcal{H} $. A conjugation on $ \mathcal{H} $ is an antilinear operator $ \mathcal{C} : \mathcal{H} \to \mathcal{H}$ which satisfies $ \langle \mathcal{C}x, \mathcal{C}y \rangle_{\mathcal{H}}=\langle y, x \rangle_{\mathcal{H}} $ for all $ x, y\in\mathcal{H} $ and $ \mathcal{C}^2=I_{\mathcal{H}} $, where $ I_{\mathcal{H}} $ is the identity operator on $ {\mathcal{H}} $. An operator $ T\in\mathcal{B}({\mathcal{H}}) $ is said to be complex symmetric if there exists a conjugation $ \mathcal{C} $ on $ \mathcal{H} $ such that $ T=\mathcal{C}T^*\mathcal{C} $. In this case, we say that $ T $ is complex symmetric with conjugation $ \mathcal{C} $. The exploration of complex symmetric operators was pioneered by Garcia, Putinar, and Wogen \cite{Garcia-Putinar-TAMS-2006,Garcia-Putinar-TAMS-2007,Garcia-Putinar-TAMS-2008,Garcia-Wogen-JFA-2009,Garcia-Wogen-TAMS-2010}. The class of complex symmetric operators  comprises a diverse set of operators, including normal operators, binormal operators, operators of algebraic degree two, Hankel operators, compressed Toeplitz operators, and the Volterra integration operator. In recent times, extensive research efforts have been directed towards investigating complex symmetric composition operators that operate on classical Hilbert spaces of analytic functions (see \cite{Bourdon-Noor-JMAA-2015,Garcia-Hammond-OPAP-2013,Jun-Kim-Ko-Lee-JFA-2014,Lim-Khoi-JMAA-2018,Noor-Severiano-PAMS-2020,Yao-JMAA-2017} and references therein).\vspace{1.2mm}

If $ \phi(0)=0 $, the following result characterizes the spectrum of the compact $ \mathcal{C}_1 $-symmetric weighted composition-differentiation operator $ D_{n, \psi, \phi}$.
\begin{thm}\cite{Han-Wang-CAOT-2021}
	Let $ n\in\mathbb{N}^{+} $. Let $ \psi\in\mathcal{H}(\mathbb{D}) $ be not identically zero and $ \phi $ be a non-constant analytic self-map of $ \mathbb{D} $ such that $ D_{n, \psi, \phi}$ is compact on $ \mathcal{A}^2 $. Suppose that $ \phi(0)=0 $ and $ D_{n, \psi, \phi}$ is $ \mathcal{C}_1 $-symmetric, then 
	\begin{align*}
		\sigma(D_{n, \psi, \phi})=\{0\}\cup\{b_n m(m-1)\cdots(m-(n-1))a_1^{m-n}/n! : m=n, n+1, \ldots\}.
	\end{align*}
Moreover, $ f_m(z)=z^m $ is an eigenvector of $ D_{n, \psi, \phi}$  with respect to the eigenvalue $ b_n m(m-1)\cdots(m-(n-1))a_1^{m-n}/n! $ for each $ m\in\mathbb{N}^{+} $ with $ m\geq n $.
\end{thm}
If $ \phi^{\prime}(0)=0 $, the following result characterizes the spectrum of the compact $ \mathcal{C}_1 $-symmetric weighted composition-differentiation operator $ D_{n, \psi, \phi}$.
\begin{thm}\cite{Han-Wang-CAOT-2021}
Let $ n\in\mathbb{N}^{+} $. Let $ \psi\in\mathcal{H}(\mathbb{D}) $ be not identically zero and $ \phi $ be a non-constant analytic self-map of $ \mathbb{D} $ such that $ D_{n, \psi, \phi}$ is bounded on $ \mathcal{A}^2 $. Suppose that $ \phi^{\prime}(0)=0 $ and $ D_{n, \psi, \phi}$ is $ \mathcal{C}_1 $-symmetric, then 
\begin{align*}
\sigma(D_{n, \psi, \phi})=\begin{cases}
\{0\}\cup\{\psi^{(n)}(a_0)\},\; \mbox{when}\; \psi^{(n)}(a_0)\neq 0\\
\{0\},\;\;\;\;\;\;\;\;\;\;\;\;\;\;\;\;\;\;\;\;\; \mbox{when}\; \psi^{(n)}(a_0)= 0.
\end{cases}
\end{align*}
Moreover, if $ \psi^{(n)}(a_0)\neq 0 $, then $ \psi $ is an eigenvector of $ D_{n, \psi, \phi} $ with respect to the eigenvalue $ \psi^{(n)}(a_0) $.
\end{thm}

We obtain the following result which characterizes the spectra of a class of composition-differentiation operators of the form $ L_{n, \psi, \phi} $.
\begin{thm}\label{Th-4.1}
Let $ \phi : \mathbb{D}\to \mathbb{D} $ be an analytic self-maps on $ \mathbb{D} $ with $ \phi(0)=0 $ and $ \psi_k\in\mathcal{H}(\mathbb{D}) $ (where $ k=1, 2, \ldots, n $) be not identically zero such that 
\begin{align*}
	L_{n, \psi, \phi}:=\sum_{k=1}^{n}c_kD_{k, \psi_k, \phi}
\end{align*}
is compact on $ \mathcal{A}^2_{\alpha}(\mathbb{D}) $. Suppose that $ L_{n,\psi, \phi} $ is complex symmetric on $ \mathcal{A}^2_{\alpha}(\mathbb{D}) $ with conjugation $ C_{\mu, \eta} $, then
\begin{align*}
	\sigma(L_{n, \psi, \phi})&=\{0\}\bigcup\bigg\{\sum_{k=1}^{m} a_kc_k\frac{m!}{(m-k)}c^{m-k}: m=1, 2,\ldots, n-1\bigg\}\\&\quad\bigcup\bigg\{\sum_{k=1}^{n}a_kc_k\frac{m!}{(m-k)}c^{m-k}: m=n, n+1, n+2, \ldots\bigg\}.
\end{align*}
Moreover, $ z^m $ is an eigenvector of $ L_{n, \psi, \phi} $ w.r.t. the eigenvalue $ \sum_{k=1}^{n}a_kc_k\frac{m!}{(m-k)}c^{m-k} $ for each $ m\in\mathbb{N} $ in case of when $m<n $ while $z^m $ is an eigenvector of $ L_{n, \psi, \phi} $ with respect to the eigenvalue $ \sum_{k=1}^{n}a_kc_k\frac{m!}{(m-k)}c^{m-k} $ for each $ m\in\mathbb{N} $ with $m\geq n $.
\end{thm}
As a consequence of Theorem \ref{Th-4.1}, we obtain the following corollary for differentiation-composition operator for space $ \mathcal{A}^2_{\alpha}(\mathbb{D}) $.
\begin{cor}
Let $ n\in\mathbb{N}^{+} $. Let $ \psi\in\mathcal{H}(\mathbb{D}) $ be not identically zero and $ \phi $ be a non-constant analytic self-map of $ \mathbb{D} $ such that $ D_{n, \psi, \phi}$ is compact on $ \mathcal{A}^2_{\alpha}(\mathbb{D}) $. Suppose that $ \phi(0)=0 $ and $ D_{n, \psi, \phi}$ is complex symmetric on $ \mathcal{A}^2_{\alpha}(\mathbb{D}) $ with conjugation $ C_{\mu, \eta} $, then 
\begin{align*}
	\sigma(D_{n, \psi, \phi})=\{0\}\cup\{a_n\frac{m!}{(m-n)!}c^{m-n} : m=n, n+1, \ldots\}.
\end{align*}
Moreover, $ f_m(z)=z^m $ is an eigenvector of $ D_{n, \psi, \phi}$  with respect to the eigenvalue $ a_n\frac{m!}{(m-n)}c^{m-n} $ for each $ m\in\mathbb{N}^{+} $ with $ m\geq n $.
\end{cor}
\begin{proof}[\bf Proof of Theorem \ref{Th-4.1}]
	Since $ L_{n,\psi, \phi} $ is complex symmetric on $ \mathcal{A}^2_{\alpha}(\mathbb{D}) $ with conjugation $ C_{\mu, \eta} $ and $ \phi(0)=0 $, in view of Theorem \ref{thm-2.1}, we obtain
	\begin{align*}
		\psi_(z)=a_kz^k \; \mbox{for}\; k=1, 2, \ldots, n\; \mbox{and}\; \phi(z)=cz
	\end{align*}
	for some $ c\in\mathbb{C} $ and $ a_1, a_2, \ldots, a_n\in\mathbb{C}\setminus\{0\} $. A simple computation shows that
\begin{align*}
	L_{n, \psi, \phi}&=\sum_{k=1}^{n}c_kD_{k, \psi_k, \phi(z^m)}\\&=\sum_{k=1}^{n}c_k\psi_k(z)(z^m)^{(k)}(\phi(z))\\&=\sum_{k=1}^{n}c_ka_kz^k(z^m)^{(k)}(cz)\\&=
	\begin{cases}
		\displaystyle\sum_{k=1}^{m}c_ka_kz^kc^{m-k}m(m-1)\cdots(m-(k-1))z^{m-k}\; \mbox{if}\; m=1, 2, \ldots, n-1\vspace{2mm}\\ 
		\displaystyle	\sum_{k=1}^{n}c_ka_kz^kc^{m-k}m(m-1)\cdots(m-(k-1))z^{m-k}\; \mbox{if}\; m=n, n+1,n+2, \ldots
	\end{cases}
\\&=
\begin{cases}
	\displaystyle\sum_{k=1}^{m}c_ka_kc^{m-k}\frac{m!}{(m-k)!}z^{m}\; \mbox{if}\; m=1, 2, \ldots, n-1\vspace{2mm}\\ 
	\displaystyle	\sum_{k=1}^{n}c_ka_kc^{m-k}\frac{m!}{(m-k)!}z^{m}\; \mbox{if}\; m=n, n+1,n+2, \ldots 
\end{cases}
\end{align*}
for all $ m\in\mathbb{N} $. Thus we see that
\begin{align*}
	&\bigg\{\sum_{k=1}^{m} a_kc_k\frac{m!}{(m-k)!}c^{m-k}: m=1, 2, \ldots, n-1\bigg\}\\&\bigcup\bigg\{\sum_{k=1}^{n} a_kc_k\frac{m!}{(m-k)!}c^{m-k}: m=n, n+1,n+2, \ldots\bigg\}
\end{align*}
belongs to the spectrum of $ L_{n, \psi, \phi} $.\vspace{1.2mm} 

Let $ \lambda $ be an arbitrary eigenvalue of $ L_{n, \psi, \phi} $ with the corresponding eigenvector $ f $. Then a simple computations shows that 
\begin{align}\label{E-4.1}
	\lambda f(z)&=\sum_{k=1}^{n}c_kD_{k, \psi_k, \psi f}\\&=\nonumber\sum_{k=1}^{n}c_k\psi_k(z)f^{(k)}(\phi(z))\\&=\sum_{k=1}^{n}c_ka_kz^kf^{(k)}(cz).\nonumber
\end{align}
If $ f(0)\neq 0 $, then it is easy to see that $ \lambda=0 $. If $ f(0)=0 $, then differentiating both sides of \eqref{E-4.1}, we obtain
\begin{align}\label{E-4.2}
	\lambda f^{\prime}(z)=\sum_{k=1}^{n}c_ka_k\left(kz^{k-1}f^{(k)}(cz)+z^kcf^{(k+1)(cz)}\right).
\end{align}
If $ f^{\prime}(0)\neq 0 $, then we see that $ \lambda=a_1c_1$. If $ f^{\prime}(0)=0 $, then differentiating $ 2 $-times of \eqref{E-4.1} gives 
\begin{align*}
	\lambda f^{\prime\prime}(z)&=\sum_{k=1}^{n}a_kz^k\sum_{j=0}^{2}\binom{2}{j}(z^k)^{2-j}\left(f^{(k)}(cz)\right)^j\\&=\sum_{k=1}^{n}a_kc_k\left(\binom{2}{0}(z^k)^{\prime\prime}f^{(k)}(cz)+\binom{2}{1}(z^k)^{\prime}f^{(k+1)}(cz)+\binom{2}{2}(z^k)f^{(k+2)}(cz)\right).
\end{align*}
If $ f^{\prime\prime}(0)\neq 0 $, then it is easy to see that
\begin{align*}
	\lambda=\sum_{k=1}^{2}a_kc_k \binom{2}{k}k!c^{2-k}=\sum_{k=1}^{2}a_kc_k\frac{2!}{(2-k)!}c^{2-k}.
\end{align*}
If $ f^{\prime\prime}(0)=0 $, then differentiating \eqref{E-4.1} $ 3 $-times, we obtain 
\begin{align*}
	\lambda f^{\prime\prime\prime}(z)&=\sum_{k=1}^{n}a_kc_k\sum_{j=1}^{3}\binom{3}{j}(z^k)^{3-j}\left(f^{(k)}(cz)\right)^{(j)}\\&=\sum_{k=1}^{n}a_kc_k\bigg(\binom{3}{0}(z^k)^{\prime\prime\prime}f^{(k)}(cz)+\binom{3}{1}(z^k)^{\prime\prime}f^{(k+1)}(cz)+\binom{3}{2}(z^k)^{\prime}f^{(k+2)}(cz)\\&\quad+\binom{3}{3}(z^k)f^{(k+3)}(cz)\bigg).
\end{align*}
If $ f^{\prime\prime\prime}(0)\neq 0 $, then we see that
\begin{align*}
	\lambda=\sum_{k=1}^{3}a_kc_k\binom{3}{k}k!c^{3-k}=\sum_{k=1}^{3}a_kc_k\frac{3!}{(3-k)!}c^{3-k}.
\end{align*}
If $ f^{\prime\prime\prime}(0)=0 $, then differentiating both sides of \eqref{E-4.1} $ 4 $-times, we obtain

\begin{align*}
	\lambda f^{(iv)}(z)&=\sum_{k=1}^{n}a_kc_k\sum_{j=1}^{4}\binom{4}{j}(z^k)^{4-j}\left(f^{(k)}(cz)\right)^{(j)}\\&=\sum_{k=1}^{n}a_kc_k\bigg(\binom{4}{0}(z^k)^{(4)}f^{(k)}(cz)+\binom{4}{1}(z^k)^{\prime\prime\prime}f^{(k+1)}(cz)+\binom{4}{2}(z^k)^{\prime\prime}f^{(k+2)}(cz)\\&\quad+\binom{4}{3}(z^k)^{\prime}f^{(k+3)}(cz)+\binom{4}{4}(z^k)f^{(k+4)}(cz)\bigg).
\end{align*}
If $ f^{(iv)}(0)\neq 0 $, then we see that
\begin{align*}
	\lambda=\sum_{k=1}^{4}a_kc_k\binom{4}{k}k!c^{4-k}=\sum_{k=1}^{4}a_kc_k\frac{4!}{(4-k)!}c^{4-k}.
\end{align*}
Continuing the process, it can be easily shown that if $ f^{(n-2)}(0)=0 $, then differentiating both sides of \eqref{E-4.1} $ (n-1) $-times, we have
\begin{align*}
\lambda f^{(n-1)}(z)&=\sum_{k=1}^{n}a_kc_k\sum_{j=0}^{n-1}\binom{n-1}{j}(z^k)^{(n-1-j)}\left(f^{(k)}(cz)\right)^{(j)}\\&=\sum_{k=1}^{n}a_kc_k\bigg(\binom{n-1}{0}(z^k)^{(n-1)}f^{(k)}(cz)+\binom{n-1}{1}(z^k)^{(n-2)}f^{(k+1)}(cz)+ \cdots \\&\quad+\binom{n-1}{n-3}(z^k)^{\prime\prime}f^{(k+n-3)}(cz)+\binom{n-1}{n-2}(z^k)^{\prime}f^{(k+n-2)}(cz)\\&\quad+\binom{n-1}{n-1}(z^k)f^{(k+n-1)}(cz)\bigg).
\end{align*}
If $ f^{(n-1)}(0)\neq 0 $, then 
\begin{align*}
	\lambda=\sum_{k=1}^{n-1}a_kc_k\binom{n-1}{k}k!c^{n-1-k}=\sum_{k=1}^{n-1}a_kc_k\frac{(n-1)!}{(n-1-k)!}c^{n-1-k}.
\end{align*}
If $ f^{(n-1)}(0)=0 $, then differentiating both sides of\eqref{E-4.1} $ n $-times , we obtain
\begin{align*}
	\lambda f^{(n)}(z)&=\sum_{k=1}^{n}a_kc_k\sum_{j=0}^{n}\binom{n}{j}(z^k)^{(n-j)}\left(f^{(k)}(cz)\right)^{(j)}\\&=\sum_{k=1}^{n}a_kc_k\bigg(\binom{n}{0}(z^k)^{(n)}f^{(k)}(cz)+\binom{n}{1}(z^k)^{(n-1)}f^{(k+1)}(cz)+ \ldots+ \\&\quad+\binom{n}{n-2}(z^k)^{\prime\prime}f^{(k+n-2)}(cz)+\binom{n}{n-1}(z^k)^{\prime}f^{(k+n-1)}(cz)\\&\quad+\binom{n}{n}(z^k)f^{(k+n)}(cz)\bigg).
\end{align*}
If $ f^{(n)}(0)\neq 0 $, then
\begin{align*}
	\lambda=\sum_{k=1}^{n}a_kc_k\binom{n}{k}k!c^{n-k}=\sum_{k=1}^{n}a_kc_k\frac{(n)!}{(n-k)!}c^{n-k}.
\end{align*}
If $ f^{(n)}(0)=0 $, then differentiating both sides of\eqref{E-4.1} $ (n+1) $-times , we obtain
\begin{align*}
		\lambda f^{(n+1)}(z)&=\sum_{k=1}^{n+1}a_kc_k\sum_{j=0}^{n+1}\binom{n+1}{j}(z^k)^{(n+1-j)}\left(f^{(k)}(cz)\right)^{(j)}\\&=\sum_{k=1}^{n+1}a_kc_k\bigg(\binom{n+1}{0}(z^k)^{(n+1)}f^{(k)}(cz)+\binom{n+1}{1}(z^k)^{(n)}f^{(k+1)}(cz)\\&\quad+\binom{n+1}{2}(z^k)^{(n-1)}f^{(k+2)}(cz)+ \cdots+ \binom{n}{n-2}(z^k)^{\prime\prime}f^{(k+n-2)}(cz)\\&\quad+\binom{n}{n-1}(z^k)^{\prime}f^{(k+n-1)}(cz)+\binom{n}{n}(z^k)f^{(k+n)}(cz)\bigg).
\end{align*}
If $ f^{(n+1)}(0)\neq 0 $, then 
\begin{align*}
	\lambda=\sum_{k=1}^{n}a_kc_k\binom{n+1}{k}k!c^{n+1-k}=\sum_{k=1}^{n}a_kc_k\frac{(n+1)!}{(n+1-k)!}c^{n+1-k}.
\end{align*}
Similarly, it can be shown that if $ f^{(n+2)}(0)\neq 0 $, then 
\begin{align*}
	\lambda=\sum_{k=1}^{n}a_kc_k\binom{n+2}{k}k!c^{n+2-k}=\sum_{k=1}^{n}a_kc_k\frac{(n+2)!}{(n+2-k)!}c^{n+2-k}.
\end{align*}
Evidently, if $ f^{(m)}(0)\neq 0 $ where $ m\geq n $, then 
\begin{align*}
	\lambda=\sum_{k=1}^{n}a_kc_k\binom{m}{k}k!c^{m-k}=\sum_{k=1}^{n}a_kc_k\frac{m!}{(m-k)!}c^{m-k}.
\end{align*}
Hence, the spectrum of $ L_{n, v, \psi} $ is
\begin{align*}
	\sigma(L_{n, v, \psi})&=\{0\}\bigcup\bigg\{\sum_{k=1}^{m} a_kc_k\frac{m!}{(m-k)}c^{m-k}: m=1, 2, \ldots, n-1\bigg\}\\&\quad\bigcup\bigg\{\sum_{k=1}^{n}a_kc_k\frac{m!}{(m-k)}c^{m-k}: m=n, n+1, n+2, \ldots\bigg\}.
\end{align*} 
This completes the proof.
\end{proof}


\section{Declaration}
\noindent\textbf{Compliance of Ethical Standards}\\

\noindent\textbf{Conflict of interest} The authors declare that there is no conflict  of interest regarding the publication of this paper.\vspace{1.5mm}

\noindent\textbf{Data availability statement}  Data sharing not applicable to this article as no datasets were generated or analysed during the current study.

\end{document}